\documentclass[11pt,a4paper]{article}
\usepackage{graphics}
\usepackage{fancyhdr}
\usepackage{mathrsfs}
\usepackage{amsfonts,amssymb,amsmath,amsthm}
\usepackage{authblk}

\usepackage[unicode,pdftex]{hyperref}
\hypersetup{
	colorlinks = true,
	citecolor = green,
	anchorcolor = blue,
	linkcolor = blue}

\newtheorem{thm}{Theorem}[section]
\newtheorem{lemma}[thm]{Lemma}

\newtheorem{prop}[thm]{Proposition}
\newtheoremstyle{rem}{10pt}{10pt}{\rmfamily}{}{\bfseries}{.}{.5em}{} 
\theoremstyle{rem}
\newtheorem{rem}[thm]{Remark}

\numberwithin{equation}{section} 

\title{Strichartz type estimates of the Airy equation}
\author[1]{Jie Chen}
\author[2]{Fan Gu\footnote{Corresponding author. E-mail address: gufan@amss.ac.cn}}
\author[3]{Boling Guo}
\affil[1]{\scriptsize \textit{School of Science, Jimei University, Xiamen 361021, P.R. China}}
\affil[2]{\scriptsize \textit{School of Statistics and Mathematics, Central University of Finance and Economics,  Beijing 102206, P.R. China}}
\affil[3]{\scriptsize \textit{Institute of Applied Physics and Computational Mathematics, Beijing 100088, PR China}}
\date{}

\begin{document}
	\maketitle
	\begin{abstract}
		In this article, we show the necessary and sufficient conditions for the inequality $\|u\|_{L_t^qL_x^r}\lesssim \|u\|_{X^{s,b}}$, 
		where $$\|u\|_{X^{s,b}}:=\|\hat{u}(\tau,\xi)\langle \xi\rangle^s\langle \tau-\xi^3\rangle^b \|_{L_{\tau,\xi}^2}. $$Here, we provide a complete classification of the indices relationships for which this inequality holds true. Such estimates will be very useful in solving the well-posedness for low regularity well-posedness of the Korteweg--de Vries equations and stochastic Korteweg--de Vries equations.
	\end{abstract}
		
		
		
	\section{Introduction}
	The Airy equation $\partial_t u+\partial_{xxx}u = 0$ is a basic dispersive model. It is also the linear part of the short wave equation, the Kowteweg--de Vries equation (KdV):
	$$\partial_t u +\partial_{xxx}u = 6u\partial_x u.$$
	
	The classical Strichartz estimate of the Airy equation is
	\begin{equation}\label{classicalstri}
		\|D_x^{1/q} e^{-t\partial_{xxx}}u_0\|_{L_t^q L_x^r}\lesssim \|u_0\|_{L^2},\quad \frac{2}{q}+\frac{1}{r} = \frac{1}{2},~4\leq q\leq \infty,
	\end{equation}
	where $D_x^{1/q}$ is the $1/q$-derivative defined by $D_x^{1/q} = \mathscr{F}^{-1}|\xi|^{1/q}\mathscr{F}$.
	Such type estimates were firstly discovered by Strichartz in \cite{strichartz} for the wave equation. There are many works on this type of inequality. See \cite{kenig1991oscillatory} and reference therein for the Airy equation and general dispersive equations. In 1993, Bourgain constructed the function space $X^{s,b}$ in \cite{bourgain1993fourier} to solve the KdV equation. The norm on $X^{s,b}$ is defined by
	$$\|u\|_{X^{s,b}} = \|\hat{u}(\tau,\xi)\langle\xi\rangle^s\langle\tau-\xi^3\rangle^b\|_{L^2_{\tau,\xi}}$$
	where $\hat{u} = \mathscr{F}_{t,x}u$ is the time-space Fourier transform of $u$, and $\langle x\rangle = \sqrt{1+x^2}$.
	
	The basic relationship between the Strichartz estimate and the norm $X^{s,b}$ is
	$$\|D_x^s e^{-t\partial_{xxx}}u_0\|_{L_t^q L_x^r}\lesssim \|u_0\|_{L^2}\Longrightarrow \|u\|_{L_t^qL_x^r}\lesssim \|u\|_{X^{-s,b}},~b>1/2.$$
	We refer to Lemma 2.9 in \cite{tao2006nonlinear}. Such estimates were very useful in the local well-posedness theory nonlinear dispersive equations by using the Bourgain norm method. See for example \cite{bourgain1993fourier}.
	
	When dealing with the well-posedness of stochastic KdV equations, the inequality $\|u\|_{L_t^qL_x^r}\lesssim \|u\|_{X^{s,b}}$ with $b<1/2$ is very useful in the bilinear or multilinear estimates for nonlinear terms. This is because the time regularity of Brownian motions are $(1/2)-$. Thus, we can only use the $X^{s,b}$, $b<1/2$ to deal with stochastic convolutions terms in the mild solutions of stochastic KdV equations. See for example \cite{de1999white}.
	
	In this paper we give a complete result for such inequality.
	\begin{thm}\label{mainres}
		Let $0<q,r\leq \infty$, $s,b\in \mathbb{R}$. Then the inequality
		\begin{equation}\label{str}
			\|u\|_{L_t^qL_x^r}\lesssim \|u\|_{X^{s,b}}
		\end{equation}
		holds if and only if $2\leq q,r\leq \infty$,
		$$b\geq \frac{1}{2}-\frac{1}{q},~ s\geq \max\left\{\frac{1}{2r}-\frac{1}{4},\frac{1}{2}-\frac{3}{q}-\frac{1}{r}\right\}, ~ s+3b \geq 2-\frac{3}{q}-\frac{1}{r},$$
		and $(q,b)\neq (\infty,1/2)$, $(r,s,b)\neq (\infty,1/2,1/2-1/q)$, $(q,r,s)\neq (\infty,\infty,1/2)$ $(q,r,s) = (2,\infty,1/2-3b)$, $(s,b)\neq (1/2-3/q-1/r,1/2)$.
	\end{thm}
	\begin{rem}
		Similar results for the Schr\"{o}dinger equation have been obtained in \cite{chen2024}. Although the Airy equation does not have the Galilean invariance, we can obtain the restriction $s\geq 1/(2r)-1/4$ by the modulation argument. In fact, by the Galilean invariance, one can obtain the restriction $s\geq 0$ for related estimates of the Schr\"{o}dinger equation. See (2.1) in \cite{chen2024}. Here we construct functions $v_N$ in \eqref{modulateasolution} by scaling and modulation for a fixed linear solution. Then we show the lower bounds of these functions in Lebesgue spaces, which can give the restriction $s\geq 1/(2r)-1/4$.
	\end{rem}
	
	\section{The necessary part of Theorem \ref{mainres}}
	Assume that \eqref{str} holds. By the same argument for Lemma 2.1 in \cite{chen2024}, we have $q,r\geq 2$, $b\geq 1/2-1/q$ and  $b>1/2$ when $q = \infty$.
	
	Let $\varphi\in C_0^\infty(\mathbb{R})$ be radial decay, and $\varphi|_{[-1,1]} = 1$, $\varphi|_{[-5/4,5/4]^c} = 0$.  $\psi(x) = \varphi(x)-\varphi(2x)$.
	
	For any $N\geq 1$, by choosing $u_N = \mathscr{F}^{-1}_{\tau,\xi}(\psi(\tau/N^3)\psi(\xi/N))$, we have $\|u_N\|_{X^{s,b}}\lesssim N^{s+3b+2}$ and $\|u_N\|_{L_t^qL_x^r}\sim N^{4-3/q-1/r}$. Thus $s+3b+2\geq 4-3/q-1/r$ which means $s+3b\geq 2-3/q-1/r$.
	
	By choosing $\tilde{u}_N = \mathscr{F}^{-1}_{\tau,\xi}(\psi(\tau-\xi^3)\psi(\xi/N))$, we have $\|\tilde{u}_N\|_{X^{s,b}}\lesssim N^{s+1/2}$.
	$$u(t,x) = cN^4\int_{\mathbb{R}^2}\eta(N^3\tau+N^3\xi^3)\psi(\xi)e^{itN^3\tau+iNx\xi}~d\xi d\tau.$$
	For $|t|\ll N^{-3}$, $|x|\ll N^{-1}$, we have $|u(t,x)|\gtrsim N$, which means that $\|u\|_{L_t^qL_x^r}\gtrsim N^{1-3/q-1/r}$. Thus we obtain $s\geq 1/2-3/q-1/r$.
	
	Let
	\begin{equation}\label{modulateasolution}
		v_{N}(t,x) = \varphi(t)\int_{\mathbb{R}}\varphi(N^{1/2}(\xi-N))e^{ix\xi}e^{it\xi^3}~d\xi.
	\end{equation}
	We have $\|v_N\|_{X^{s,b}}\sim N^{s-1/4}$, and
	\begin{align*}
		\|v_N\|_{L_t^qL_x^r} & = \left\|\varphi(t)\int_{\mathbb{R}}\psi(N^{1/2}\xi)e^{ix\xi}e^{it\xi^2(\xi+3N)}~d\xi\right\|_{L_t^qL_x^r}\\
		& = N^{\frac{1}{2r}+\frac{3}{2q}-\frac{1}{2}}\left\|\varphi(N^{3/2}t)\int_{\mathbb{R}}\psi(\xi)e^{ix\xi}e^{it\xi^2(\xi+3N^{3/2})}~d\xi\right\|_{L_t^qL_x^r}.
	\end{align*}
	Let 
	$$A(t,x):=\int_{\mathbb{R}}\psi(\xi)e^{ix\xi}e^{it\xi^2(\xi+3N^{3/2})}~d\xi$$
	Note that by the Plancherel identity, one has $\|A(t,x)\|_{L^2_x}\sim \|\psi\|_{L^2}\sim 1$. Since $|\partial_{\xi\xi}(x\xi+t\xi^2(\xi+3N^{3/2}))|\sim N^{3/2}|t|$ on the support set of $\psi$, by the stationary phase argument (c.f. \cite{stein2011functional}, Section 8.2) one has $|A(t,x)|\lesssim \langle N^{3/2}t\rangle^{-1/2}$. If $|x|\gg \langle N^{3/2}t\rangle$, by the non-stationary phase argument (see also \cite{stein2011functional}, Section 8.2) we have $|A(t,x)|\lesssim |x|^{-2}$. Thus 
	\begin{align*}
		\|A(t,x)\|_{L^1_x}&\lesssim \int_{|x|\lesssim \langle N^{3/2}t\rangle}\langle N^{3/2}t\rangle^{-1/2}~dx+\int_{|x|\gg \langle N^{3/2}t\rangle}|x|^{-2}~dx\\
		& \lesssim  \langle N^{3/2}t\rangle^{\frac{1}{2}}.
	\end{align*}
	By the H\"{o}lder inequality, we have
	\begin{align*}
		1\sim\|A(t,x)\|_{L_x^2}&\leq \|A(t,x)\|_{L^1_x}^{\frac{r-2}{2r-2}}\|A(t,x)\|_{L^r_x}^{\frac{r}{2r-2}}\\
		&\lesssim \langle N^{3/2}t\rangle^{\frac{r-2}{4r-4}} \|A(t,x)\|_{L^r_x}^{\frac{r}{2r-2}},
	\end{align*}
	which implies $\|A(t,x)\|_{L^r_x}\gtrsim \langle N^{3/2} t\rangle^{-1/2+1/r}$.
	Then
	\begin{align*}
		\|v_N\|_{L_t^q L_x^r} \gtrsim N^{\frac{1}{2r}+\frac{3}{2q}-\frac{1}{2}}\left\|\varphi(N^{3/2}t)\langle N^{3/2} t\rangle^{-\frac{1}{2}+\frac{1}{r}}\right\|_{L_t^q} \sim N^{\frac{1}{2r}-\frac{1}{2}}.
	\end{align*}
	Thus $s\geq 1/(2r)-1/4$.
	
	\begin{lemma}
		If \eqref{str} holds for some $q,r,s,b$, then $(r,s)\neq (\infty,1/2-3b)$, $(q,r,s)\neq (\infty,\infty,1/2)$, $(r,s,b)\neq (\infty,1/2,1/2-1/q)$, $(s,b)\neq (1/2-3/q-1/r,1/2)$.
	\end{lemma}
	\begin{proof}[\textbf{Proof}]
	    Note that $b\geq 0$. If $\|u\|_{L_t^qL_x^\infty}\lesssim \|u\|_{X^{s,b}}$. Then by choosing $u(t,x) = f(t)g(x)$, we have
		\begin{equation}\label{fenli}
			\|f\|_{L^q}\|g\|_{L^\infty}\lesssim \|u\|_{X^{s,b}}\lesssim \|f\|_{H^b}\|g\|_{H^s}+\|f\|_{H^b}\|g\|_{H^{s+3b}}.
		\end{equation}
		Thus $s+3b>1/2$. 
		
		If $\|u\|_{L_{t,x}^\infty}\lesssim \|u\|_{X^{s,b}}$, by choosing $u(t,x) = \eta(t)e^{-t\partial_{xxx}}u_0$ for $\eta\in C_0^\infty(\mathbb{R})$ with $\eta|_{[0,1]}\equiv 1$ and $u_0\in H^s$, we have  $\|u_0\|_{L_x^\infty}\lesssim \|e^{-t\partial_{xxx}}u_0\|_{L_t^\infty(0,1;L^\infty_x)}\lesssim \|u\|_{X^{s,b}}\lesssim \|u_0\|_{H^s}$. Thus by the Sobolev inequality we obtain $s>1/2$.
		
		If $\|u\|_{L_t^qL_x^\infty}\lesssim \|u\|_{X^{1/2,1/2-1/q}}$, by \eqref{fenli} we have
		\begin{align*}
			\|f\|_{L_t^q}\|g\|_{L_x^\infty}\lesssim \|f\|_{H_t^{\frac{1}{2}-\frac{1}{q}}}\|g\|_{H_x^{\frac{1}{2}}}+\|f\|_{L_t^2}\|g\|_{H_x^{2-\frac{3}{q}}}.
		\end{align*}
		If $q = 2$, we obtain $\|g\|_{L^\infty_x}\lesssim \|g\|_{H_x^{1/2}}$ which is a contradiction. If $2<q<\infty$, we choose a sequence $f_n$ such that $\|f_n\|_{L^q_t}\sim \|f_n\|_{H_t^{1/2-1/q}}$ and $\|f_n\|_{L^2_t}\rightarrow 0$, then we also have $\|g\|_{L^\infty_x}\lesssim \|g\|_{H_x^{1/2}}$ which is a contradiction.
		
		If $\|u\|_{L_t^qL_x^r}\lesssim \|u\|_{X^{1/2-3/q-1/r,1/2}}$, firstly we have $q,r\geq 2$. Choose $u$ such that $\mathrm{supp}(\hat{u})\subset \{(\tau,\xi):|\xi|\in [N,2N],1\leq |\tau-\xi^3|\leq N^3\}$. Let $\hat{v}(\tau,\xi) = \hat{u}(N^3\tau,N\xi)$, $\hat{v}(\tau,\xi) = \chi_{1\leq |\xi|\leq 2}f(\tau-\xi^3)$. By the same argument for Lemma 2.2 in \cite{chen2024}, we obtain $\|f(\tau)\|_{L^1_\tau}\lesssim  \|\tau^{1/2}f(\tau)\|_{L_\tau^2}$, $\forall$ $f\geq 0$, $\mathrm{supp}(f)\subset [0,1]$. We obtain a contradiction. This finishes the proof.
	\end{proof}
	Combining all the conditions in this section, we conclude the proof of the necessary part of our main result.
	
	\section{The sufficient part of Theorem \ref{mainres}}
	Let $P_N = \mathscr{F}^{-1}\chi_{\langle\xi\rangle\sim N}\mathscr{F}$, $Q_L = \mathscr{F}^{-1}_{\tau,\xi}\chi_{\langle \tau-\xi^3\rangle \sim L}\mathscr{F}_{t,x}$. The arguments for Lemmas 3.1--3.2 in \cite{chen2024} rely on the Plancherel identity, Minkowski and Sobolev inequalities. Thus they also work here. We have the following proposition.
	\begin{prop}\label{ul}
		$\|u\|_{L_t^qL_x^2}\lesssim \|u\|_{X^{0,1/2-1/q}}$, $2\leq q<\infty$. $\|u\|_{L_t^\infty L_x^r}\lesssim \|u\|_{X^{1/2-1/r,b}}$, $2\leq r< \infty$, $b>1/2$. $\|u\|_{L_{t,x}^\infty}\lesssim \|u\|_{X^{s,b}}$, $s>1/2$, $b>1/2$. $\|u\|_{L_t^2L_x^r}\lesssim \|u\|_{X^{1/2-1/r,0}}$, $2< r<\infty$.
	\end{prop}
	\begin{prop}\label{q2}
		$\|u\|_{L_t^2L_x^r}\lesssim \|u\|_{X^{1/(2r)-1/4,1/4-1/(2r)}}$, $2< r<\infty$.
	\end{prop}
	\begin{proof}[\textbf{Proof}]
		Firstly we show
		\begin{equation}\label{singlemodu}
			\|P_NQ_Lu\|_{L_t^2L_x^r}\lesssim L^{\frac{1}{4}-\frac{1}{2r}}N^{\frac{1}{2r}-\frac{1}{4}}\|Q_LP_N u\|_{L_{t,x}^2},~~2\leq r<\infty.
		\end{equation}
		The proof is similar to the one in \cite{chen2024}. Choose a Schwartz function $\phi$ such that, $\mathrm{supp}(\hat{\phi})\subset [-1,1]$. $\phi|_{[-1,1]}\geq \chi_{[-1,1]}$. Define $\phi_{L,k}(t):=\phi(Lt-k)$, $k\in \mathbb{Z}$. By the Strichartz estimate \eqref{classicalstri} and the transference principle, we have 
		\begin{align*}
			\|P_NQ_Lu\|_{L_t^2L_x^r}^2&\leq \sum_{k} \|\phi_{L,k}(t)(Q_{L}P_Nu)\|_{L_{t\in [kL^{-1},(k+1)L^{-1}]}^2L_x^r}^2\\
			&\lesssim \sum_k L^{\frac{1}{2}-\frac{1}{r}-1}\|\phi_{L,k}(t)(Q_{L}P_Nu)\|_{L_t^{\frac{4r}{r-2}}L_x^r}^2\\
			&\lesssim \sum_k L^{\frac{1}{2}-\frac{1}{r}}N^{\frac{1}{r}-\frac{1}{2}}\|\phi_{L,k}(t)Q_{L}u\|_{L_{t,x}^2}^2\\
			&\sim L^{\frac{1}{2}-\frac{1}{r}} N^{\frac{1}{r}-\frac{1}{2}}\|Q_L u\|_{L_{t,x}^2}^2.
		\end{align*}
		Thus, we obtain \eqref{singlemodu}. 
		
		By choosing $0<\epsilon<\min\{1,r/2-1\}$ we have
		\begin{equation}\label{litobi}
			\begin{aligned}
				\|P_Nu\|_{L_t^2L_x^r}^2 &= \|(P_N u)^2\|_{L_t^1L_x^\frac{r}{2}}\lesssim \sum_{L_1\leq L_2} \|Q_{L_1}P_NuQ_{L_2}P_Nu\|_{L_t^1 L_x^{\frac{r}{2}}}\\
				&\lesssim \sum_{L_1\leq L_2}\|Q_{L_1}P_Nu\|_{L_t^2L_x^{\frac{r}{1-\epsilon}}}\|Q_{L_2}P_Nu\|_{L_t^2L_x^{\frac{r}{1+\epsilon}}}\\
				&\lesssim N^{\frac{1}{r}-\frac{1}{2}}\sum_{L_1\leq L_2}L_1^{\frac{1}{4}-\frac{1-\epsilon}{2r}}L_2^{\frac{1}{4}-\frac{1+\epsilon}{2r}}\|Q_{L_1}u\|_{L_{t,x}^2}\|Q_{L_2}u\|_{L_{t,x}^2}\\
				&\lesssim N^{\frac{1}{r}-\frac{1}{2}}\sum_{L}L^{2(\frac{1}{4}-\frac{1}{2r})}\|Q_{L}u\|_{L_{t,x}^2}^2\sim \|P_Nu\|_{X^{\frac{1}{2r}-\frac{1}{4},\frac{1}{4}-\frac{1}{2r}}}^2.
			\end{aligned}
		\end{equation}
		By the Littlewood--Paley theory and the Minkowski inequality, we have
		\begin{align*}
			\|u\|_{L_t^2L_x^r}&\sim \|P_N u\|_{L_t^2L_x^rl^2_N}\\
			&\lesssim \|P_N u\|_{l^2_N L_t^2L_x^r}\lesssim \|P_Nu\|_{l^2_NX^{\frac{1}{2r}-\frac{1}{4},\frac{1}{4}-\frac{1}{2r}}}\lesssim \|u\|_{X^{\frac{1}{2r}-\frac{1}{4},\frac{1}{4}-\frac{1}{2r}}}.
		\end{align*}
		We finish the proof.
	\end{proof}
	
	\begin{prop}\label{midcase}
		Assume $2<q,r<\infty$,
		$$b\geq \frac{1}{2}-\frac{1}{q},~ s\geq \max\left\{\frac{1}{2r}-\frac{1}{4},\frac{1}{2}-\frac{3}{q}-\frac{1}{r}\right\}, ~ s+3b \geq 2-\frac{3}{q}-\frac{1}{r}$$
		and $(s,b)\neq (1/2-3/q-1/r,1/2)$. Then \eqref{str} holds.
	\end{prop}
	\begin{proof}[\textbf{Proof}]
		Given $2<q,r<\infty$, due to the Sobolev inequality and Lemma \ref{ul} we have $\|u\|_{L_t^qL_x^r}\lesssim \|\langle D_x \rangle^{1/2-1/r}u\|_{L_t^qL_x^2}\lesssim \|u\|_{X^{1/2-1/r,1/2-1/q}}$.
		
		By the Strichartz estimate one has
		\begin{equation*}
			\|Q_LP_Nu\|_{L_t^{\frac{4r}{r-2}}L_x^r}\lesssim L^{\frac{1}{2}}N^{-\frac{r-2}{4r}}\|Q_LP_N u\|_{L_{t,x}^2},~~2< r<\infty.
		\end{equation*}
		By interpolation with \eqref{singlemodu}, we have
		$$\|Q_LP_Nu\|_{L_t^qL_x^r}\lesssim L^{\frac{3}{4}-\frac{1}{q}-\frac{1}{2r}}N^{-\frac{r-2}{4r}}\|Q_LP_N u\|_{L_{t,x}^2},~~2< r<\infty,~2\leq q\leq \frac{4r}{r-2}.$$
		
		If $2<q<4r/(r-2)$, for $s = 1/(2r) - 1/4$, $b = 3/4-1/q-1/(2r)$, $0<\varepsilon\ll 1$ then we have
		\begin{equation*}
			\begin{aligned}
				\|P_N u\|_{L_t^qL_x^r}^2 &= \|(P_Nu)^2\|_{L_{t}^\frac{q}{2}L_x^\frac{r}{2}}\lesssim \sum_{L_1\leq L_2}\|Q_{L_1}P_N u Q_{L_2}P_N u\|_{L_{t}^\frac{q}{2}L_x^\frac{r}{2}}\\
				&\lesssim \sum_{L_1\leq L_2}\|Q_{L_1}P_N u \|_{L_{t}^{\frac{q}{1-\varepsilon}}L_x^r}\|Q_{L_2}P_N u\|_{L_{t}^\frac{q}{1+\varepsilon} L_x^r} \\
				&\lesssim N^{2s}\sum_{L_1\leq L_2}L_1^{b+\frac{\varepsilon}{q}}L_2^{b-\frac{\varepsilon}{q}}\|Q_{L_1}P_N u \|_{L_{t,x}^{2}}\|Q_{L_2}P_N u\|_{L_{t,x}^2}\\
				&\lesssim N^{2s} \sum_{L}L^{2b}\|Q_{L}P_N u\|_{L_{t,x}^2}^2\sim \|P_Nu\|_{X^{s,\frac{3}{4}-\frac{1}{q}-\frac{1}{2r}}}^2
			\end{aligned}
		\end{equation*}
		Then by the Littlewood--Paley theory, we conclude that $\|u\|_{L_t^qL_x^r}\lesssim \|u\|_{X^{s,b}}$ for $s = 1/(2r)-1/4$, $b = 3/4-1/q-1/(2r)$.
		
		If $q\geq 4r/(r-2)$, by the Sobolev inequality and the Strichartz estimates \eqref{classicalstri}, for any $b>1/2$ we have
		\begin{equation*}
			\|u\|_{L_t^qL_x^r}\lesssim \|\langle  D_x\rangle^{\frac{1}{2}-\frac{2}{q}-\frac{1}{r}}u\|_{L_t^qL_x^{\frac{2q}{q-4}}}\lesssim \|\langle D_x\rangle^{\frac{1}{2}-\frac{2}{q}-\frac{1}{r}}u\|_{X^{-\frac{1}{q},b}}\lesssim \|u\|_{X^{{\frac{1}{2}-\frac{3}{q}-\frac{1}{r}},b}}.
		\end{equation*}
		
		For $q\geq 4r/(r-2)$, let $s = 2-3/q-1/r-3b$ with $1/2-1/q<b<1/2$. By the Sobolev inequality, we only need to prove related estimates for $q = 4r/(r-2)$, $2<r<\infty$. By the Bernstein inequality and the Sobolev inequality and Lemma \ref{ul} we have $$\|Q_LP_Nu\|_{L_t^\frac{4r}{r-2}L_x^r}\lesssim N^{\frac{1}{2}-\frac{1}{r}}L^{\frac{1}{4}+\frac{1}{2r}}\|Q_{L}P_N u\|_{L^2_{t,x}}.$$  
		Thus, by the Strichartz estimate and the interpolation, we have 
		$$\|P_N u\|_{L_t^{4r/(r-2)}L_x^r}\lesssim \|P_Nu\|_{X^{s,b}}.$$
		Then by the Littlewood--Paley theory, we have $\|u\|_{L_t^{4r/(r-4)}L_x^r}\lesssim \|u\|_{X^{s,b}}$ for $(s,b) = (5/4+1/(2r)-3b,b)$ and $0<1/2-b\ll 1$. By the interpolation and Proposition \ref{midcase}, we conclude that $\|u\|_{L_t^{4r/(r-2)}L_x^r}\lesssim \|u\|_{X^{5/4+1/(2r)-3b,b}}$, $1/2-1/q\leq b<1/2$. This finishes the proof.
	\end{proof}
	\begin{prop}\label{rinfty}
		Let $2\leq q <\infty$.  $\|u\|_{L_t^qL_x^\infty}\lesssim \|u\|_{X^{s,1/2-1/q}}$, $s>1/2$. 
		
		If $4\leq q<\infty$, then $\|u\|_{L_t^qL_x^\infty}\lesssim \|u\|_{X^{1/2-3/q,b}}$, $b>1/2$ and $\|u\|_{L_t^qL_x^\infty}\lesssim \|u\|_{X^{2-3/q-3b,b}}$, $1/2-1/q<b<1/2$.
		
		If $2<q<4$, then $\|u\|_{L_t^qL_x^\infty}\lesssim \|u\|_{X^{2-3/q-3b,b}}$, $1/2-1/q<b\leq 3/4-1/q$. $\|u\|_{L_t^2L_x^\infty}\lesssim \|u\|_{X^{-1/4,b}}$,  $b>1/4$.
	\end{prop}
	\begin{proof}[\textbf{Proof}]
		By the Sobolev inequality and Proposition \ref{ul}, we have
		\begin{align*}
			\|u\|_{L_t^qL_x^\infty}\lesssim \|\langle D_x\rangle^s u\|_{L_t^q L_x^2} \lesssim \|u\|_{X^{s,\frac{1}{2}-\frac{1}{q}}},\quad2\leq q<\infty ,~s>\frac{1}{2}.
		\end{align*}
		By the oscillatory integrals argument, one has $\|D_x^{6/q-1}\mathscr{F}_\xi^{-1}(e^{i\xi^3})\|_{L_x^\infty}\lesssim 1$, $4\leq q<\infty$. See for example \cite{kenig1991oscillatory}. Then by the scaling, we have $\|D_x^{6/q-1}\mathscr{F}_\xi^{-1}(e^{it\xi^3})\|_{L_x^\infty}\lesssim |t|^{-2/q}$. From this decay estimate, it is standard to obtain (see for example \cite{tao2006nonlinear}) $\|e^{-t\partial_{xxx}}u_0\|_{L_t^qL_x^\infty}\lesssim \|u_0\|_{H^{1/2-3/q}}$. Then $\|u\|_{L_t^qL_x^\infty}\lesssim \|u\|_{X^{1/2-3/q,b}}$ for $4\leq q<\infty$, $b>1/2$.

		Let $b = 1/2-1/q +\epsilon$, $s = 2-3/q-3b$ where $0<\epsilon\ll 1$. By the Bernstein inequality, Propositions \ref{ul}--\ref{midcase}, we have
		\begin{equation}\label{frequdeco}
			\begin{aligned}
				&\quad\|Q_{L_1}uQ_{L_2}u\|_{L_t^{\frac{q}{2}}L_x^\infty}\\
				&\lesssim \sum_{N\lesssim M}\|Q_{L_1}P_N u\|_{L_t^{\frac{q}{1-2q\epsilon}}L_x^\infty}\|Q_{L_2}P_M u\|_{L_t^{\frac{q}{1+2q\epsilon}}L_x^\infty}\\
				&\quad+\sum_{N\gg M}\|Q_{L_1}P_N u\|_{L_t^qL_x^\infty}\|Q_{L_2}P_M u\|_{L_t^qL_x^\infty}\\
				&\lesssim\sum_{N\lesssim M}N^{\frac{1}{2}}M^{\frac{1+2q\epsilon}{q}}\|Q_{L_1}P_N u\|_{L_t^{\frac{q}{1-2q\epsilon}}L_x^2}\|Q_{L_2}P_M u\|_{L_{t,x}^{\frac{q}{1+2q\epsilon}}}\\
				&\quad+\sum_{N\gg M}M^{\frac{1}{2}}N^{\frac{1}{q}}\|Q_{L_1}P_N u\|_{L_{t,x}^{q}}\|Q_{L_2}P_M u\|_{L_t^qL_x^2}\\
				&\lesssim\sum_{N\lesssim M}N^{\frac{1}{2}}M^{\frac{1}{2}-6\epsilon}L_1^{b+\epsilon}L_2^{b-\epsilon}\|Q_{L_1}P_N u\|_{L_{t,x}^2}\|Q_{L_2}P_M u\|_{L_{t,x}^{2}}\\
				&\quad+\sum_{N\gg M}M^{\frac{1}{2}}N^{\frac{1}{2}-6\epsilon}L_1^{b+\epsilon}L_2^{b-\epsilon}\|Q_{L_1}P_N u\|_{L_{t,x}^2}\|Q_{L_2}P_M u\|_{L_{t,x}^2}\\
				&\lesssim (L_1L_2)^{b}(L_1/L_2)^{\epsilon}\|Q_{L_1}u\|_{X^{1/2-3\epsilon,0}}\|Q_{L_2}u\|_{X^{1/2-3\epsilon,0}}.
			\end{aligned}
		\end{equation}
		By the Strichartz estimate, we also have $$\|Q_{L_1}uQ_{L_2}u\|_{L_t^{\frac{q}{2}}L_x^\infty}\lesssim \|Q_{L_1}u\|_{X^{\frac{1}{2}-\frac{3}{q},\frac{1}{2}}}\|Q_{L_2}u\|_{X^{\frac{1}{2}-\frac{3}{q},\frac{1}{2}}}.$$
		By the interpolation, for any $1/2-1/q<b<1/2$, $s = 2-3/q-3b$, there exists $\delta>0$ such that
		\begin{align*}
			\|Q_{L_1}uQ_{L_2}u\|_{L_t^\frac{q}{2}L_x^\infty}\lesssim (L_1/L_2)^\delta\|Q_{L_1}u\|_{X^{s,b}}\|Q_{L_2}u\|_{X^{s,b}}.
		\end{align*}
		By the same argument for \eqref{litobi}, we obtain $\|u\|_{L_t^qL_x^\infty}\lesssim \|u\|_{X^{2-3/q-3b,b}}$, for $4\leq q<\infty$, $1/2-1/q<b<1/2$.
		
		If $q = 2$, we slightly modify the proof of \eqref{singlemodu}. By the Strichartz estimate \eqref{classicalstri} we have
		\begin{align*}
			\|Q_Lu\|^2_{L_t^2L_x^\infty}&\leq \sum_{k} \|\phi_{L,k}(t)(Q_{L}u)\|_{L^2_{t\in [kL^{-1},(k+1)L^{-1}]}L_x^\infty}^2\\
			&\lesssim \sum_k L^{-\frac{1}{2}}\|\phi_{L,k}(t)(Q_{L}u)\|_{L_t^4L_x^\infty}^2\\
			&\lesssim \sum_k L^{\frac{1}{2}}\|\langle D_x\rangle^{-\frac{1}{4}}\phi_{L,k}(t)(Q_{L}u)\|_{L_{t,x}^2}^2 \sim L^{\frac{1}{2}} \|Q_L \langle D_x\rangle^{-\frac{1}{4}}u\|_{L_{t,x}^2}^2.
		\end{align*}
		Then for any $b>1/4$,
		\begin{align*}
			\|u\|_{L_t^2L_x^\infty}\lesssim \|Q_Lu\|_{L_t^2L_x^\infty l_{L}^1} \lesssim \|L^\frac{1}{4}Q_{L}\langle D_x\rangle^{-\frac{1}{4}}u\|_{l_{L}^1L_{t,x}^2}\lesssim \|u\|_{X^{-\frac{1}{4},b}}.
		\end{align*}
		For $2<q<4$, by choosing $0<\epsilon\leq \min\{(q-2)/2,(4-q)/4\}$, we have
		\begin{align*}
			\|u\|^2_{L_t^qL_x^\infty}&\lesssim \sum_{L_1\leq L_2}\|Q_{L_1}uQ_{L_2}u\|_{L_t^{\frac{q}{2}}L_x^\infty} \lesssim \sum_{L_1\leq L_2}\|Q_{L_1}u\|_{L^{\frac{q}{1-\epsilon}}_tL_x^\infty}\|Q_{L_2}u\|_{L_t^\frac{q}{1+\epsilon}L_x^\infty}\\
			&\lesssim \sum_{L_1\leq L_2} \|Q_{L_1}u\|^{1-\theta_1}_{L_t^2L_x^\infty}\|Q_{L_1}u\|^{\theta_1}_{L_t^4L_x^\infty}\|Q_{L_2}u\|_{L_t^2L_x^\infty}^{1-\theta_2}\|Q_{L_2}u\|_{L_t^4L_x^\infty}^{\theta_2}\\
			&\lesssim \sum_{L_1\leq L_2} (L_1L_2)^{\frac{3}{4}-\frac{1}{q}}\left(\frac{L_1}{L_2}\right)^\frac{\epsilon}{q}\|Q_{L_1}\langle D_x\rangle^{-\frac{1}{4}}u\|_{L_{t,x}^2}\|Q_{L_2}\langle D_x\rangle^{-\frac{1}{4}}u\|_{L_{t,x}^2}\\
			&\lesssim \|u\|_{X^{-\frac{1}{4},\frac{3}{4}-\frac{1}{q}}}^2
		\end{align*}
		where $\theta_1 = 2-4(1-\epsilon)/q$, $\theta_2 = 2-4(1+\epsilon)/q$. Note that the proof of \eqref{frequdeco} is also available for $2<q<4$. Thus one has $\|u\|_{L_t^qL_x^\infty}\lesssim \|u\|_{X^{2-3/q-3b,b}}$ for $0<b-(1/2-1/q)\ll 1$. By the interpolation, we finish the proof.
	\end{proof}
	
	Combining Propositions \ref{ul}--\ref{rinfty}, we conclude the proof of the sufficient part of Theorem \ref{mainres}.
	
	\section*{Acknowledgment}
The authors are  grateful to the referees for valuable comments to improve the paper.
	
\section*{Declarations}
\subsection*{Funding}
	Jie Chen was supported in part by the NSFC, grants 12301116, 12450001.
\subsection*{Ethical approval} Not applicable.
\subsection*{Informed consent} Not applicable.
\subsection*{Author Contributions} Jie Chen, Fan Gu, Boling Guo have the same contribution in the topic design, manuscript writing, and revisions.
\subsection*{Data Availability Statement} Not applicable.
\subsection*{Conflict of Interest} The authors declare no conflict of interest.
\subsection*{Clinical trial number} Not applicable.

	\phantomsection 
	\bibliographystyle{amsplain}
	\addcontentsline{toc}{section}{References}
	\bibliography{reference}

%
\end{document}